\documentclass[amssymb,twoside,12pt]{article}
\usepackage{latexsym,amsmath,graphicx,mathrsfs,amssymb}
\usepackage{amsfonts}
\usepackage{enumitem}
\newtheorem{theorem}{Theorem}[section]

\numberwithin{equation}{section}
\newenvironment{proof}{\indent{\em Proof:}}{\quad \hfill
$\Box$\vspace*{2ex}}

\begin{document}
\setcounter{page}{1}
\begin{center}
\vspace{0.3cm} {\large{\bf Mild and strong solutions for Hilfer evolution equation}} \\
\vspace{0.4cm}
 J. Vanterler da C. Sousa $^1$ \\
vanterler@ime.unicamp.br  \\

\vspace{0.30cm}
Leandro S. Tavares $^2$\\
leandro.tavares@ufca.edu.br\\

\vspace{0.30cm}
E. Capelas de Oliveira $^3$\\
capelas@ime.unicamp.br\\
\vspace{0.35cm}

\vspace{0.30cm}
$^{1,3}$ Department of Applied Mathematics, Imecc-Unicamp,\\ 13083-859, Campinas, SP, Brazil.\\
$^{2}$  Center of Sciences and Technology, Federal University of Cariri,Juazeiro do Norte, CE, CEP: 63048-080, Brazil.
\end{center}

\def\baselinestretch{1.0}\small\normalsize
\begin{abstract}
In this paper, we investigate the existence and uniqueness of mild and strong solutions of fractional semilinear evolution equations in the Hilfer sense, by means of Banach fixed point theorem and the Gronwall inequality.
\end{abstract}
\allowdisplaybreaks
\section{Introduction}

In recent years, many researchers have looked in particular at the field of fractional calculus, especially for fractional differential equations. It is already more consolidated and proven that in fact, investigating the properties of solutions of fractional differential equations, seems to be better than the integer case. In addition, there is some sense the modeling, in the  fractional setting, several species , which in turn has also been of great value, because it is possible to obtain more consistent results  with respect to the reality \cite{mainardi}. Investigating the existence, uniqueness and stability of mild, strong and classical solutions of fractional differential equations, has gained prominence and strength in the scientific community. Due to the previous facts,the fractional calculus produced important and high quality papers, see for instance \cite{chauhan,kexue}. 

In 2012, Shu and Wang \cite{shuwang} investigated the existence and uniqueness of mild solution for non-local fractional differential equations with non-local conditions using the Banach fixed point theorem. In 2016, Shu and Shi \cite{shi1} performed the work on the expressions obtained so far that were related to mild solutions to impulsive fractional evolution equations. In the following year, Gou and Li \cite{gou} investigated the existence of mild solution in global and local context, for impulsive semilinear integral equations in the fractional sense with non-compact semigroup in spaces Banach, in which the authors emphasized the importance and effectiveness of these fractional integro-differential equations has in preexisting problems. In this sense, many other works have been published and investigated. We suggest some works \cite{shu,zhou,herza}.

Motivated by the works \cite{shi1,gou,princ}, in this paper, we consider the fractional semilinear evolution equation
\begin{equation}\label{3-4}
\left\{ 
\begin{array}{rll}
^{H}\mathbb{D}_{t_{0+}}^{\alpha ,\beta }\xi\left( t\right) +\mathcal{A}\xi\left( t\right)  & = & 
\phi\left( t,\xi\left( \sigma \left( t\right) \right) \right)  \\ 
I_{t_{0+}}^{1-\gamma }\xi\left( t_{0}\right) +\varphi\left(
t_{1},t_{2},...,t_{p},\xi\left( \cdot \right) \right)  & = & \xi_{0}
\end{array}
\right. 
\end{equation}
where $\mathcal{A}$ is the infinitesimal generator of a $C_{0}$ semigroup $\mathbb{F}\left( t\right) _{t\geq 0}$ on a Banach space $\Lambda$, $^{H}\mathbb{D}_{t_{0+}}^{\alpha ,\beta }\left( \cdot \right) $ is the Hilfer fractional derivative of order $\alpha \left( 0<\alpha <1\right) $ and type $\beta \left( 0\leq \beta \leq 1\right) $, $I_{t_{0+}}^{1-\gamma }\left( \cdot \right) $ is the Riemann-Liouville fractional integral of order $1-\gamma \left( \gamma =\alpha +\beta \left( 1-\beta \right) \right) $, $0\leq t_{0}<t_{1}<\cdot \cdot \cdot \cdot <t_{p}\leq t_{0}+a,$ $a>0,$ $\xi_{0}\in X$ and $\phi:\left[ t_{0},t_{0}+a\right] \times \Lambda\rightarrow \Lambda,$ $\varphi:\left[ t_{0},t_{0}+a\right] ^{p}\times \Lambda\rightarrow \Lambda$ and $\sigma :\left[ t_{0},t_{0}+a\right] \rightarrow \left[ t_{0},t_{0}+a\right] $ are given functions.

The motivation of this work, besides investigating the properties of the mild and strong solutions, is to provide to the many researchers that investigate the results on the existence and uniqueness of several types of fractional differential equations, new results that allow to further strengthen the field as well as provide a range of tools and news. In this sense, we investigate the existence and uniqueness of mild and strong solutions for Eq.(\ref{3-4}), using the Banach fixed point theorem and the Gronwall inequality.
\section{Preliminaries}

Let $n-1<\alpha \leq n$ with $n\in \mathbb{N}$ and $f,\psi \in C^{n}\left([a,b],\mathbb{R}\right) $ be two functions such that $\psi $ is increasing and $\psi ^{\prime }\left( t\right) \neq 0$, for all $t\in J.$ The left-sided $\psi$-Hilfer fractional derivative $^{H}\mathbb{D}_{0+}^{\alpha ,\beta }\left( \cdot \right) $ of a function $f$ of order $\alpha $ and type $0\leq \beta \leq 1$ is defined by \cite{ZE1,leibniz}
\begin{equation*}
^{H}\mathbb{D}_{a+}^{\alpha ,\beta;\psi }\xi\left( t\right) =I_{a+}^{\beta \left( n-\alpha
\right) ;\psi }\left( \frac{1}{\psi ^{\prime }\left( t\right) }\frac{d}{dt}%
\right) ^{n}I_{a+}^{\left( 1-\beta \right) \left( n-\alpha \right) ;\psi
}\xi\left( t\right),
\end{equation*}
where $I^{\alpha}_{a+}(\cdot)$ is $\psi$-Riemann-Liouville fractional integral. 

Let $\phi:I \rightarrow X$. Consider the fractional initial value problem
\begin{equation}\label{1-2}
\left\{ 
\begin{array}{cll}
^{H}\mathbb{D}_{t_{0+}}^{\alpha ,\beta }\xi\left( t\right)  & = & \mathcal{A}\xi\left( t\right)
+\phi\left( t\right) ,\text{ }t\in \left( t_{0},t_{0}+a\right]  \\ 
I_{t_{0+}}^{1-\gamma }\xi\left( t_{0}\right)  & = & \xi_{0}.
\end{array}
\right. 
\end{equation}

\begin{theorem} Consider  the following conditions:
\begin{enumerate}
    \item $\Lambda$ is a reflexive Banach space and $\mathcal{A}$ is the infinitesimal generator of $\mathbb{F_{\alpha,\beta}}\left( t\right) _{t\geq 0}$ on $\Lambda$;
    
    \item $\phi:\Lambda\rightarrow \Lambda$ is Lipschitz continuous on $I$ and $\xi_{0}\in D\left( \mathcal{A}\right)$;

\end{enumerate}

The {\rm Eq.(\ref{1-2})} has a unique strong solution $\xi$ on $I$ given by the formula {\rm \cite{sousanew}}
\begin{equation*}
\xi\left( t\right) =\mathbb{F}_{\alpha ,\beta }\left( t-t_{0}\right) \xi_{0}+\int_{t_{0}}^{t}\mathcal{K}_{\alpha }\left( t-s\right) \phi\left( s\right) ds,\text{ }t\in I.
\end{equation*}
\end{theorem}

The mild solution for the nonlocal Cauchy problem Eq.(\ref{3-4}) on $I$ in the sense of Hilfer fractional derivative, is given by means of the integral equation
\begin{equation*}
\xi\left( t\right) =\mathbb{F}_{\alpha ,\beta }\left( t-t_{0}\right) \xi_{0}-\mathbb{F}_{\alpha ,\beta }\left( t-t_{0}\right) \varphi\left( t_{1},t_{2},...,t_{p},\xi\left( \cdot \right) \right) +\int_{t_{0}}^{t}\mathcal{K}_{\alpha }\left( t-s\right) \phi\left( s,\xi\left( \sigma \left( s\right) \right) \right) ds,\text{ }t\in I.
\end{equation*}

A function $\xi$ is said to be a strong solution of problem {\rm Eq.(\ref{3-4})} on $I$ if $\xi$ is differentiable a.e. on $I$ $^{H}\mathbb{D}_{t_{0+}}^{\alpha ,\beta }\in \left( L^{1}\left( \left( t_{0},t_{0}+a \right] ,X\right) \right)$ and satisfies Eq.(\ref{3-4}).

\section{Main results}
In this section, we will investigate the existence and uniqueness of mild and strong solutions for the fractional evolution equation introduced by means of the Hilfer fractional derivative. In order to obtain the main results of the paper, we will use the Banach fixed point theorem and the Gronwall inequality.

Before investigating the main results of this paper, consider some conditions:

\begin{enumerate}
        \item $0\leq t_{0}<t_{1}<\cdot \cdot \cdot <t_{p}\leq t_{0}+a$ and $B_{R}:=\left\{ \mu:\left\Vert \mu\right\Vert \leq R\right\} \subset \lambda$;
    
    \item $\sigma :I\rightarrow I$ is absolutely continuous and $\exists b>0$ a constant such that $\sigma ^{\prime }\left( t\right) \geq b$ for $t\in I$;
    
    \item $\varphi:I^{p}\times \Lambda\rightarrow \Lambda$ and $\exists \lambda>0$ a constant such that
\begin{equation*}
\left\Vert \varphi\left( t_{1},t_{2},...,t_{p},\xi_{1}\left( \cdot \right) \right) -g\left( t_{1},t_{2},...,t_{p},\xi_{2}\left( \cdot \right) \right) \right\Vert _{C_{1-\gamma }}\leq \lambda\left\Vert \xi_{1}-\xi_{2}\right\Vert _{C_{1-\gamma }}
\end{equation*}
$\xi_{1},\xi_{2}\in C_{1-\gamma }\left( I,B_{R}\right) $ and $\varphi\left( t_{1},t_{2},...,t_{p}\right) \in D\left( \mathcal{A}\right)$;

    \item $-\mathcal{A}$ is the infinitesimal generator of a $C_{0}$ semigroup $\mathbb{F}\left( t\right) _{t\geq 0}$ on $\Lambda$;
    
    \item $\zeta_{1}=\underset{t\in \left[ 0,a\right] }{\max }\left\Vert \mathbb{F}_{\alpha ,\beta }\left( t\right) \right\Vert _{C_{1-\gamma }}$, $\zeta_{2}=\underset{s\in I}{\max }\left\Vert \phi\left( s,0\right) \right\Vert _{C_{1-\gamma }}$ and \\ $\zeta_{3}=\underset{\xi\in C_{1-\gamma }\left( I,B_{R}\right) }{\max }\left\Vert \varphi\left( t_{1},t_{2},....,t_{p},v\left( \cdot \right) \right) \right\Vert _{C_{1-\gamma }}$

    \item $\zeta_{1}\left( \left\Vert \xi_{0}\right\Vert _{C_{1-\gamma }}+\zeta_{3}+\left( ar\delta/b\right) +a\zeta_{2}\right) \leq r$ and $\zeta_{1}\lambda+\left(\zeta_{1}\delta a/b\right) <1$.
    \end{enumerate}

\begin{theorem} Assume {\rm (1)-(6)} and consider the following conditions:
\begin{enumerate}
    \item $\Lambda$ is a Banach space with norm $\left\Vert {(\cdot)}\right\Vert _{C_{1-\gamma }}$ and $\xi_{0}\in \Lambda$;
    
        \item $\phi:I\times \Lambda\rightarrow \Lambda$ is continuous in $t$ on $I$ and $\exists \delta>0$ constant such that
\begin{equation*} 
\left\Vert \phi\left( s,\mu_{1}\right) -\phi\left( s,\mu_{2}\right) \right\Vert _{C_{1-\gamma }}\leq \delta\left\Vert \mu_{1}-\mu_{2}\right\Vert _{C_{1-\gamma }}, \text{for s} \in I \text{and}\, \mu_{1},\mu_{2}\in B_{R}.
\end{equation*}

        \end{enumerate}
Then problem {\rm Eq.(\ref{3-4})} has a unique mild solution on $I$.
\end{theorem}

\begin{proof} To realize the proof, we consider $\Omega:=C_{1-\gamma }\left( I,B_{R}\right)$ and define the following operator $\mathcal{F}$ on $\Omega$, given by
\begin{equation*}
\begin{aligned}
\left( \mathcal{F}\mu\right) \left( t\right) &=\mathbb{F}_{\alpha ,\beta }\left( t-t_{0}\right)
\xi_{0}-\mathbb{F}_{\alpha ,\beta }\left( t-t_{0}\right) \varphi\left( t_{1},t_{2},...,t_{p},\mu \left( \cdot \right) \right)\\ &+\int_{t_{0}}^{t}\mathcal{K}_{\alpha }\left( t-s\right) \phi\left( s,\mu\left( \sigma \left( s\right) \right) \right) ds.
\end{aligned}
\end{equation*}

Then, by definition of the norm in $\Omega$, we get
\begin{equation*}
\begin{aligned}
\left\Vert \left( \mathcal{F}\mu\right) \left( t\right) \right\Vert _{C_{1-\gamma }} &\leq \left\Vert \mathbb{F}_{\alpha ,\beta }\left( t-t_{0}\right) \right\Vert
_{C_{1-\gamma }}\left\Vert \xi_{0}\right\Vert _{C_{1-\gamma }}\\
&+ \left\Vert
\mathbb{F}_{\alpha ,\beta }\left( t-t_{0}\right) \right\Vert _{C_{1-\gamma
	}}\left\Vert \varphi\left( t_{1},t_{2},...,t_{p},\mu\left( \cdot \right) \right)
	\right\Vert _{C_{1-\gamma }} \\
& +\int_{t_{0}}^{t}\left\Vert \mathcal{K}_{\alpha }\left( t-s\right) \right\Vert
_{C_{1-\gamma }}\left\Vert \phi\left( s,\mu\left( \sigma \left( s\right) \right)
\right) \right\Vert _{C_{1-\gamma }}ds\\
& \leq \zeta_{1}\left\Vert \xi_{0}\right\Vert _{C_{1-\gamma
	}}+\zeta_{1}\zeta_{3}+\zeta_{1}\int_{t_{0}}^{t}\left\Vert \phi\left( s,\mu\left( \sigma \left( s\right)
	\right) \right) \right\Vert _{C_{1-\gamma }}ds\\
&\leq \zeta_{1}\left\Vert \xi_{0}\right\Vert _{C_{1-\gamma
	}}+\zeta_{1}\zeta_{3}+\zeta_{1}\delta\int_{t_{0}}^{t}\left\Vert \mu\left( \sigma \left( s\right) \right)
	\right\Vert _{C_{1-\gamma }}\left( \frac{\sigma ^{\prime }\left( s\right) }{b
	}\right) ds+\zeta_{1}\zeta_{2}\int_{t_{0}}^{t}ds\\
&\leq \zeta_{1}\left\Vert \xi_{0}\right\Vert _{C_{1-\gamma }}+\zeta_{1}\zeta_{3}+\frac{\zeta_{1}\delta r}{b}\left(
\sigma \left( t\right) -\sigma \left( t_{0}\right) \right) +\zeta_{1}\zeta_{2}\left(
t-t_{0}\right)\\
&=\zeta_{1}\left[ \left\Vert \xi_{0}\right\Vert _{C_{1-\gamma }}+\zeta_{3}+\frac{\delta r}{b}a+\zeta_{2}a\right] \leq r.
\end{aligned}
\end{equation*}

Therefore, $\mathcal{F}\left( \Omega \right) \subset \Omega$. Let investigate the norm, for every $\mu_{1},\mu_{2}\in \Omega$ and $t\in I$, we obtain
\begin{equation*}
\begin{aligned}
&\left\Vert \left( \mathcal{F}\mu_{1}\right) \left( t\right) -\left( \mathcal{F}\mu_{2}\right)
\left( t\right) \right\Vert _{C_{1-\gamma }}  \\
&\leq \left\Vert \mathbb{F}_{\alpha ,\beta }\left( t-t_{0}\right) \right\Vert
_{C_{1-\gamma }}\left\Vert \varphi\left( t_{1},t_{2},...,t_{p},\mu_{1}\left( \cdot
\right) \right) -\varphi\left( t_{1},t_{2},...,t_{p},\mu_{2}\left( \cdot \right)
\right) \right\Vert _{C_{1-\gamma }} \\
& \leq \zeta_{1}\zeta_{3}\left\Vert \mu_{1}-\mu_{2}\right\Vert _{C_{1-\gamma }}+\frac{\zeta_{1}\delta}{b} \int_{t_{0}}^{t}\left\Vert \mu_{1}\left( \sigma \left( s\right) \right)-\mu_{2}\left( \sigma \left( s\right) \right) \right\Vert _{C_{1-\gamma }}\left( \frac{\sigma ^{\prime }\left( s\right) }{b}\right) ds \\
&\leq \zeta_{1}\zeta_{3}\left\Vert \mu_{1}-\mu_{2}\right\Vert _{C_{1-\gamma }}+\frac{\zeta_{1}\delta}{b} \int_{\sigma \left( t_{0}\right) }^{\sigma \left( t\right) }\left\Vert \mu_{1}\left( s\right) -\mu_{2}\left( s\right) \right\Vert _{C_{1-\gamma }}ds \\
& \leq \left( \zeta_{1}\zeta_{13}+\frac{\zeta_{1}\delta a}{b}\right) \left\Vert \mu_{1}-\mu_{2}\right\Vert
_{C_{1-\gamma }}.
\end{aligned}
\end{equation*}

Now, taking $q:=\left( \zeta_{1}\zeta_{3}+\dfrac{\zeta_{1}\delta a}{b}\right)$, we have
\begin{equation*}
\left\Vert \mathcal{F}\mu_{1}-\mathcal{F}\mu_{2}\right\Vert _{C_{1-\gamma }}\leq q\left\Vert \mu_{1}-\mu_{2}\right\Vert _{C_{1-\gamma }}, \text{with}\,\, 0<q<1.
\end{equation*}

Thus, we guarantee that $\mathcal{F}$ is a contraction in the metric space $\Omega$. Then, by means of the Banach fixed point theorem for $\mathcal{F}$ in the space $\Omega$, we conclude that, in fact, this point is the mild solution of the problem {\rm Eq.(\ref{3-4})} on $I$.
\end{proof}

The second main result is to investigate the existence and uniqueness of strong solution for Eq.(\ref{3-4}). So we have the following result.

\begin{theorem} Assume {\rm (1)-(6)} and consider the following conditions:
\begin{enumerate}
    \item $\Lambda$ is a reflexive Banach space with norm $\left\Vert {(\cdot)}\right\Vert _{C_{1-\gamma }}$ and $\xi_{0}\in \Lambda$;
    
    \item $\phi:I\times \Lambda\rightarrow \Lambda$ is continuous in $t$ on $I$ and $\exists \delta>0$ a constant such that
\begin{equation*}
\left\Vert \phi\left( s_{1},\mu_{1}\right) -\phi\left( s_{2},\mu_{2}\right) \right\Vert _{C_{1-\gamma }}\leq \delta\left( \left\Vert s_{1}-s_{2}\right\Vert _{C_{1-\gamma }}+\left\Vert \mu_{1}-\mu_{2}\right\Vert _{C_{1-\gamma }}\right) 
\end{equation*}
for $s_{1},s_{2}\in I$ and $\mu_{1},\mu_{2}\in B_{R}$;
    
      \item $\xi$ is the mild solution of problem {\rm Eq.(\ref{3-4})} on $I$ and there exists a constant $\widetilde{R}>0$ such that
\begin{equation}
\left\Vert \xi\left( \sigma \left( s\right) \right) -\xi\left( \sigma \left(t\right) \right) \right\Vert _{C_{1-\gamma }}\leq \widetilde{R}\left\Vert \xi\left( s\right) -\xi\left( t\right) \right\Vert _{C_{1-\gamma }}, \text{for s,t}\in I.
\end{equation}
\end{enumerate}

Then $\xi$ is a strong solution of problem {\rm Eq.(\ref{3-4})} on $I$.
\end{theorem}

\begin{proof}
By Theorem 1, the problem Eq.(\ref{3-4}), admits a unique mild solution in $C_{1-\gamma}(I,\Lambda)$, once the conditions are satisfied. In order to obtain the existence and uniqueness of the strong solution, we will use the fact that the solution $\xi$, is mild for Eq.(\ref{3-4}) on $I$. Then, for any $t\in I$, we get
\begin{eqnarray*}
&&\left\Vert \xi\left( t+h\right) -\xi\left( t\right) \right\Vert _{C_{1-\gamma }}
\\
&&\leq \left[ \left\Vert \mathbb{F}_{\alpha ,\beta }\left( t+h-t_{0}\right)
\right\Vert _{C_{1-\gamma }}+\left\Vert \mathbb{F}_{\alpha ,\beta }\left(
t-t_{0}\right) \right\Vert _{C_{1-\gamma }}\right] \left\Vert
\xi_{0}\right\Vert _{C_{1-\gamma }} \\
&&+\left[ \left\Vert \mathbb{F}_{\alpha ,\beta }\left( t+h-t_{0}\right) \right\Vert
_{C_{1-\gamma }}+\left\Vert \mathbb{F}_{\alpha ,\beta }\left( t-t_{0}\right)
\right\Vert _{C_{1-\gamma }}\right] \left\Vert \varphi\left(
t_{1},t_{2},...,t_{p},\xi\left( \cdot \right) \right) \right\Vert
_{C_{1-\gamma }} \\
&&+\int_{t_{0}}^{t_{0}+h}\left\Vert \mathcal{K}_{\alpha }\left( t+h-s\right)
\right\Vert _{C_{1-\gamma }}\left\Vert \phi\left( s,\xi\left( \sigma \left(
s\right) \right) \right) -\phi\left( s,0\right) \right\Vert _{C_{1-\gamma }}ds
\\
&&+\int_{t_{0}}^{t_{0}+h}\left\Vert \mathcal{K}_{\alpha }\left( t+h-s\right)
\right\Vert _{C_{1-\gamma }}\left\Vert \phi\left( s,0\right) \right\Vert
_{C_{1-\gamma }}ds \\
&&+\int_{t_{0}}^{t}\left\Vert \mathcal{K}_{\alpha }\left( t-s\right) \right\Vert
_{C_{1-\gamma }}\left\Vert \phi\left( s+h,\xi\left( \sigma \left( s+h\right)
\right) \right) -\phi\left( s,u\left( \sigma \left( s\right) \right) \right)
\right\Vert _{C_{1-\gamma }}ds \\
&&\leq 2\zeta_{1}h\left\Vert \xi_{0}\right\Vert _{C_{1-\gamma }}+2\zeta_{1}\zeta_{3}h+\frac{\delta \zeta_{1}hr}{b}%
+\zeta_{1}\zeta_{2}b+\zeta_{1} \delta\int_{t_{0}}^{t}\left\Vert h\right\Vert _{C_{1-\gamma }}ds \\
&&+\zeta_{1} \delta\int_{t_{0}}^{t}\left\Vert u\left( \sigma \left( s+h\right) \right)
-\xi\left( \sigma \left( s\right) \right) \right\Vert _{C_{1-\gamma }}ds \\
&&\leq \theta+\zeta_{1}\delta R\widetilde{C}\int_{t_{0}}^{t}\left( t-s\right) ^{\alpha
-1}\left\Vert \xi\left( s+h\right) -\xi\left( s\right) \right\Vert _{C_{1-\gamma
}}ds,
\end{eqnarray*}
where $\theta:=2\zeta_{1}h\left\Vert \xi_{0}\right\Vert _{C_{1-\gamma }}+2\zeta_{1}\zeta_{3}h+\dfrac{\zeta_{1}\delta hr}{b}+\zeta_{1}\zeta_{2}b+\zeta_{1}\delta ha$.

By means of the Gronwall inequality (see\cite{gronwall}), we obtain
\begin{equation*}
\left\Vert \xi\left( t+h\right) -\xi\left( t\right) \right\Vert _{C_{1-\gamma}}\leq \theta\mathbb{E}_{\alpha }\left[ \zeta_{1}\delta R\widetilde{C}a^{\alpha }\Gamma \left( \alpha \right) \right], \text{for t}\in I.
\end{equation*}

Thus, $\xi$ is Lipschitz continuous on $I$. Note that, because $u$ is Lipschitz in $I$ and condition (iii), we have that $t\rightarrow \phi\left( t,\xi\left( t\right) \right) $ is Lipschitz continuous on $I$. In this sense, by means of Theorem 1 and Theorem 2, we have that the fractional Cauchy problem with its initial condition Eq.(\ref{3-4}), admits a unique solution in the interval $I$, which satisfies the integral equation
\begin{equation}
\left\{ 
\begin{array}{rll}
^{H}\mathbb{D}_{t_{0+}}^{\alpha ,\beta }\mu\left( t\right) +\mathcal{A}\mu\left( t\right)  & = & 
\phi\left( t,\xi\left( \sigma \left( t\right) \right) \right) ,\text{ }t\in \left[ t_{0},t_{0}+a\right]  \\ 
I_{t_{0+}}^{1-\gamma }\mu\left( t_{0}\right)  & = & \xi_{0}-\varphi\left(t_{1},t_{2},...,t_{p},\xi\left( \cdot \right) \right) 
\end{array}
\right.
\end{equation}
has a unique solution on $I$ satisfying the equation 
\begin{equation*}
\mu\left( t\right) =\mathbb{F}_{\alpha ,\beta }\left( t-t_{0}\right) \xi_{0}-\mathbb{F}_{\alpha ,\beta }\left( t-t_{0}\right) \varphi\left( t_{1},t_{2},...,t_{p},\xi\left( \cdot \right) \right) +\int_{t_{0}}^{t}\mathcal{K}_{\alpha }\left( t-s\right) \phi\left( s,\xi\left( \sigma \left( s\right) \right) \right) ds=\xi\left( t\right).
\end{equation*}

Thus, we conclude that, $\xi$ is a strong solution of fractional Cauchy problem {\rm Eq.(\ref{3-4})} in the interval $I$.
\end{proof}


\section*{Acknowledgment}
JVCS acknowledges the financial support of a PNPD-CAPES (process number nº88882.305834/2018-01) scholarship of the Postgraduate Program in Applied Mathematics of IMECC-Unicamp.


\end{document}